\documentclass[leqno,10pt]{amsart}
\usepackage{amssymb,amsmath,amsfonts,amsthm,mathtools,latexsym,ifthen,bezier}
\usepackage{graphicx}
\usepackage{amsxtra}
\usepackage{xspace}
\usepackage{dirtytalk}
\usepackage{url}
\usepackage[english]{babel} 

\usepackage{enumitem}
\usepackage[margin=1.3in]{geometry}
\usepackage{setspace}
\usepackage[english]{babel}
\usepackage{latexsym}
\usepackage{enumitem}
\usepackage{amsthm}
\usepackage{amssymb}
\DeclareMathAlphabet{\mathpzc}{OT1}{pzc}{m}{it}
\usepackage[latin1]{inputenc}
\usepackage{afterpage}

\usepackage{bbm}
\usepackage{tikz} 

\newtheorem{theorem}{Theorem}[section]

\newtheorem{lemma}[theorem]{Lemma}
\newtheorem{proposition}[theorem]{Proposition}
\newtheorem{corollary}[theorem]{Corollary}
\newtheorem{observation}[theorem]{Observation}
\newtheorem{fact}[theorem]{Fact}
\newtheorem{claim}[theorem]{Claim}
\newtheorem{maintheorem}[theorem]{Main Theorem}
\theoremstyle{definition}
\newtheorem{definition}[theorem]{Definition}
\newtheorem{example}[theorem]{Example}
\theoremstyle{remark}
\newtheorem{remark}{Remark}

\newtheorem{question}{Question}

\def\hook{\upharpoonright}
\def\forces{\Vdash}

\newfont{\ssi}{cmssi12 at 12pt}

\newenvironment{ea*}{\begin{eqnarray*}}{\end{eqnarray*}}

\setbox0=\hbox{$\longrightarrow$}

\newcommand{\calA}{\mathcal{A}}
\newcommand{\calB}{\mathcal{B}}

\newcommand{\calL}{\mathcal{L}}
\newcommand{\calM}{\mathcal{M}}
\newcommand{\calN}{\mathcal{N}}

\renewcommand{\phi}{\varphi}

\newcommand{\ZFC}{\ensuremath{\mathsf{ZFC}}\xspace}
\newcommand{\ZF}{\ensuremath{\mathsf{ZF}}\xspace}

\newcommand{\proves}{\vdash}

\def\<#1>{\langle#1\rangle}

\renewcommand{\P}{{\mathord{\mathbb P}}}

\newcommand{\MA}{\ensuremath{\mathsf{MA}}}

\newcommand{\MP}{\ensuremath{\mathsf{MP}}}

\newcommand{\ColNothing}{\mathrm{Col}}
\newcommand{\Col}[1]{\ColNothing(#1)}

\newcommand{\MPColNothing}[1]{\MP_{\Col{\dot{\kappa}}}}

\newcommand{\CH}{\ensuremath{\mathsf{CH}}\xspace}

\def\hook{\upharpoonright}
\def\forces{\Vdash}

\def\ZFC{\mathsf{ZFC}}
\def\PA{\mathsf{PA}}

\def\MA{\mathsf{MA}}
\def\MAone{\mathsf{MA}_{\aleph_1}}

\def\CH {\mathsf{CH}}

\title{Destructibility and Axiomatizability of Kaufmann Models}
\author[Switzer]{Corey Bacal Switzer}
\address[C.~B.~Switzer]{Institut f\"{u}r Mathematik, Kurt G\"odel Research Center, Universit\"{a}t Wien, Kolingasse 14-16, 1090 Wien, AUSTRIA}
\email{corey.bacal.switzer@univie.ac.at}

\thanks{\emph{Acknowledgements:} The author would like to thank the
Austrian Science Fund (FWF) for the generous support through grant number Y1012-N35. Data sharing not applicable to this article as no datasets were generated or analysed during the current study.}

\date{}

\begin{document}

\maketitle

\begin{abstract}
A Kaufmann model is an $\omega_1$-like, recursively saturated, rather classless model of $\PA$ (or $\ZF$). Such models were constructed by Kaufmann under the combinatorial principle $\diamondsuit_{\omega_1}$ and Shelah showed they exist in $\ZFC$ by an absoluteness argument. Kaufmann models are an important witness to the incompactness of $\omega_1$ similar to Aronszajn trees. In this paper we look at some set theoretic issues related to this motivated by the seemingly na\"{i}ve question of whether such a model can be ``killed" by forcing without collapsing $\omega_1$. We show that the answer to this question is independent of $\ZFC$ and closely related to similar questions about Aronszajn trees. As an application of these methods we also show that it is independent of $\ZFC$ whether or not Kaufmann models can be axiomatized in the logic $L_{\omega_1, \omega} (Q)$ where $Q$ is the quantifier ``there exists uncountably many".
\end{abstract}

\section{Introduction}
A Kaufmann model is an $\omega_1$-like, recursively saturated, rather classless model (these terms are defined below). Kaufmann first constructed such models for $\PA$ in \cite{Kauf77} under the combinatorial principle $\diamondsuit_{\omega_1}$ and noted in that paper that a similar construction works for models of $\ZF$. In \cite{Shelah2ndorderprop2} Shelah showed that Kaufmann models (for $\PA$ and $\ZF$) exist in $\ZFC$ by an absoluteness argument. These structures form an important class of models of arithmetic (and set theory) that have been extensively studied, see \cite[Chapter 10]{MOPAKS} and the references therein. There are several reasons for this. First of all a Kaufmann model represents a counterexample to the analogue of several theorems about countable recursively saturated models of $\PA$ (and $\ZF$) holding at the uncountable including most notably the fact that countable recursively saturated models of $\PA$  have inductive partial satisfaction classes, see \cite[Theorem 1.9.3, Proposition 1.9.4]{MOPAKS}. They also are a witness to set theoretic incompactness at $\omega_1$. For instance, the following is immediate from the fact that all countable, recursively saturated models of $\PA$ have satisfaction classes and Tarski's theorem on the undefinability of truth.  
\begin{proposition}
Let $\calM$ be a Kaufmann model of $\PA$. By rather classlessness $\calM$ cannot have a partial inductive satisfaction class. However, there is a club of countable elementary submodels $\calN \prec \calM$ so that $\calN$ carries a satisfaction class.
\end{proposition}
Of course the above proposition is also true for models $\ZF$ with the analogue of a ``partial inductive satisfaction class" defined appropriately, see below.

Kaufmann models are also very closely related to trees. This was used in Shelah's absoluteness proof and also features prominently in Schmerl's work on generalizations of Kaufmann models to higher cardinals \cite{Sch81}. The analogy with trees is the jumping off point for the current work. Our na\"{i}ve question that started this work was whether there could be a Kaufmann model which could be killed by forcing without collapsing $\omega_1$. Note that this is similar to asking whether there is an Aronszajn tree to which an uncountable branch can be added by forcing with out collapsing $\omega_1$. The answer in that case is independent of the axioms of $\ZFC$: if there is a Souslin tree the answer is ''yes" while if all Aronszajn trees are special the answer is ``no". In the case of Kaufmann models the answer turns out to be the same. Specifically we prove the following theorem (proved as Theorems \ref{MAthm} and \ref{diamondthm} respectively).

\begin{maintheorem}
Let $T$ be any consistent completion of either $\PA$ or $\ZF$.
\begin{enumerate}
\item
Assume $\MA_{\aleph_1}$ holds. If $\calM \models T$ is a Kaufmann model and $\mathbb P$ is a forcing notion so that $\forces_{\mathbb P}$``$\calM$ is not Kaufmann" then $\mathbb P$ collapses $\omega_1$.
\item
Assume the combinatorial principle $\diamondsuit_{\omega_1}$ holds. There is a Kaufmann model $\calM \models T$ and a Souslin tree $S$ so that forcing with $S$ adds a satisfaction class to $\calM$.
\end{enumerate}
\label{mainthm1}
\end{maintheorem}

It remains unclear whether the property of ``being destructible by $\omega_1$-preserving forcing" has a completely combinatorial or model theoretic characterization but the models used in the proof of Main Theorem \ref{mainthm1} can be used to show the following, which is the second main theorem of this paper (See Theorem \ref{mainthm22} below).

\begin{maintheorem}
Let $Q$ be the quantifier \say{there exists uncountably many ...} and $L_{\omega_1, \omega} (Q)$ be the infinitary logic $L_{\omega_1, \omega}$ enriched by this quantifier. The following hold:
\begin{enumerate}
\item
Under $\MA_{\aleph_1}$ there is an $L_{\omega_1, \omega}(Q)$ sentence $\psi$ in the language of $\PA$ (respectively $\ZF$) enriched with a single unary function symbol $f$, $\calL_{\PA} (f)$ (respectively $\calL_{\ZF}(f)$), so that a model $\calM \models \PA$ ($\calM \models \ZF$) is Kaufmann if and only if there is an expansion of $\calM$ to an $\calL_{\PA}(f)$-structure (respectively to a $\calL_{\ZF}(f)$) satisfying $\psi$.
\item
Under the combinatorial princple $\diamondsuit_{\omega_1}$ there is a Kaufmann model $\calM$ so that given any expansion $\calL'$ of the language of $\PA$ (respectively $\ZF$) and any expansion of $\calM$ to an $\calL'$-structure, $\calM '$, and any countable set $X$ of $L_{\omega_1, \omega}(Q)$ sentences in the signature $\calL'$ there is a model $\calN$ which agrees with $\calM '$ about the truth of every sentence in $X$ but carries a satisfaction class for its $\calL$-reduct. In particular, the $\calL$-reduct of $\calN$ is not rather classless.
\end{enumerate}
\label{mainthm2}
\end{maintheorem}

Informally the Main Theorem \ref{mainthm2} can be phrased as saying it is independent of $\ZFC$ if Kaufmann models can be axiomatized by an $L_{\omega_1, \omega}(Q)$ sentence. This logic is a natural one to consider in the context of such models since being $\omega_1$-like and recursively saturated are expressible here hence the question is really about (in)expressibility of rather classlessness. Moreover this logic plays an important role in Shelah's aforementioned absoluteness result, \cite[Theorem 6]{Shelah2ndorderprop2}, and is used in several other applications of abstract model theory to $\omega_1$-like structures, see \cite{L(Q)}. In fact, part 1 can be deduced as an immediate corollary of the proof of  \cite[Theorem 6]{Shelah2ndorderprop2}. I do not know if this was observed by Shelah at the time. Part 2, as far as I know, is completely new. 

A first draft of this paper focused solely on the case of models of arithmetic. The anonymous referee astutely observed that in fact many of the results presented extend to a broader class of models which can be described in terms of tree-like models satisfying the collection scheme (see \cite{KeislerTarskiSymp} and the definitions in subsection 2.1 below). In particular all the results applied to models of $\ZF$. Following the referee's suggestion, in this version we have reworded several of the results to accommodate this more general perspective. Specifically, we give the basic set up for tree-like models in section 2 and then couch proofs for the rest of the paper in terms of $\PA$ and $\ZF$. Presumably many of the ideas presented here could be applied to other ``foundational theories" including weaker set theories, however we leave the investigation of which specific theorems apply to which specific weak theories for later work. 

The rest of this paper is organized as follows. In Section 2 we give some basic definitions and background that will be used throughout. In section 3 Part 1 of Main Theorem \ref{mainthm1} is proved. In Section 4 Part 2 of Main Theorem \ref{mainthm1} is proved. In Section 5 Main Theorem \ref{mainthm2} is proved. Section 6 concludes with some open questions and lines for further research. 

\subsection*{ Acknowledgments} I would like to thank Ali Enayat, Roman Kossak, and Bartosz Wcis\l{}o for several very informative and helpful conversations relating to the material in this paper. I would also like to thank Ali Enayat for pointing out the papers \cite{KeislerTarskiSymp} and \cite{Enayatstandardmodel} to me. Finally I would like to thank the anonymous referee for many insightful comments, in particular the observation that the results here apply more broadly to tree-like models satisfying the collection scheme. 

\section{Tree-like Models and Basic Definitions}

Throughout we will be interested in the languages $\mathcal L_{\PA}$ of $\PA$, and $\calL_{\ZF}$ of $\ZF$ which for us, in the case of $\PA$, includes a symbol $\leq$ for the natural ordering definable in $\PA$. All of the results below work equally well for any countable extension of $\mathcal L_{\PA}$ (respectively $\calL_{\ZF}$) and any theory $\PA^*$, that is $\PA$ in that language with induction extended to formulas in that language (respectively $\ZF^*$, that is $\ZF$ with the comprehension and replacement schemes expanded to include all formulas in this language). Given a first order structure such as $\calM$, $\calN$, $\calM_\alpha$ etc we always let the associated non-calligraphic letter, $M$, $N$, $M_\alpha$ etc denote the universe of the model. When it won't cause confusion this won't be stated explicitly. Also, when it will not cause confusion we will refer somewhat ambiguously to a language $\calL$ which could be either $\calL_{\PA}$ or $\calL_{\ZF}$ depending on the context. Since many of the proofs work mutatis mutandis for the two theories we will prove the statements once in this general setting. For ease of notation by a {\em model} we will always mean an $\mathcal L$ structure $\calM$ modeling $\PA$ (respectively $\ZF$) unless otherwise stated. Also, throughout {\em definable} means definable with arbitrary parameters unless specified otherwise. Since we're looking at applications of set theory to model theory, and hoping to appeal to researchers in both these fields, we have included more definitions and proof sketches than usual in order to make this paper more self contained for the reader who is an expert in only one of these subjects. For all undefined terms in the model theory of $\PA$ we suggest the reader consult \cite{MOPAKS}. For set theory we recommend \cite{KenST}. 

\subsection{Tree-Like Models}

We begin with a brief overview of the vocabulary of tree-like models. Tree-like models were first introduced and studied in depth by Keisler in \cite{KeislerTarskiSymp} and the reader is encouraged to look to that article for more details and background. Recall that a {\em tree-like} order $\mathcal T = \langle T, \leq_T\rangle$ is a partial order with the additional property that given any $t \in T$ the set $R(t) := \{u \in T \; | \; u \leq_T t\}$ is linearly ordered by $\leq_T$. As is standard, when it will cause no confusion we will confuse a tree-like order with its universe $T$. A {\em tree} is a tree-like order $T$ with the additional property that $\leq_T$ is well-founded. A {\em branch} $B \subseteq T$ is a maximal, linearly ordered subset of $T$. Throughout this article we will be interested in tree-like orders with the following additional properties.
\begin{enumerate}
\item
(Rooted) There is a unique $t_0 \in T$ which is minimal with respect to $\leq_T$.
\item
(Normal) For every $t \in T$ there are $s, u \in T$ so that $t \leq_T s, u$ and $s$ and $u$ are incomparable with respect to $\leq_T$.
\end{enumerate}
Note that Item 2 implies that in particular there are no maximal elements of $T$. From now on we will assume that all of our tree-like orders are rooted and normal.

\begin{definition}
A {\em ranked tree} is a structure $\mathcal T = \langle T, \leq_T, O, \leq_O, r\rangle$ so that 
\begin{enumerate}
\item
$\langle T, \leq_T\rangle$ is a rooted, normal tree-like order,
\item
$\langle O, \leq_O\rangle$ is a linear order
\item 
$r:T \to O$ is a function, called a {\em ranking function}, so that $t \leq_T u$ implies $r(t) \leq_O r(u)$ with equality holding on the left if and only if it holds on the right.
\item
For each $t \in T$, the image of $\{u \in T \; | \; u \leq_T t \, {\rm or} \, t \leq_T u\}$ under $r$ maps onto $O$.
\end{enumerate}
\end{definition}

Given a ranked tree $\mathcal T = \langle T, \leq_T, O, \leq_O, r\rangle$ and a branch $B \subseteq T$. We say that $B$ is {\em cofinal} if its image under $r$ surjects onto $O$. Throughout this paper we will only be interested in cofinal branches and ``branch" will mean cofinal branch unless otherwise stated. If $\kappa$ is a cardinal, we say that ranked tree $\mathcal T$ is $\kappa$-{\em like} if $O$ has size $\kappa$ but for each $a \in O$ the set of $t \in T$ so that $r(t) \leq_O a$ has size ${<}\kappa$.

\begin{definition}
Let $\kappa$ be a cardinal.
\begin{enumerate}
\item
A structure $\mathcal A = \langle A, T^A, \leq_T, O^A, \leq_O, r, P_1, ...\rangle$ is called a {\em tree-like model} if $\mathcal T^A = \langle T^A, \leq_T, O^A, \leq_O, r\rangle$ is a ranked tree.
\item
If $\mathcal A$ is any $\mathcal L_0$-structure for any first-order language $\mathcal L_0$ we will equally call $\mathcal A$ a {\em tree-like model} if there are (parametrically) definable predicates $T^A$, $O^A$, definable relations $\leq_T$, $\leq_O$ and a definable function $r$ so that $\mathcal A$ enriched with this extra structure form a tree-like model in the obvious way.
\item
A tree-like model (in either sense) is $\kappa$-like if its associated ranked tree is $\kappa$-like.
\item
A tree-like model is {\em rather branchless} if every cofinal branch is definable.
\end{enumerate}
\end{definition}

For most ``foundational theories", every model can be thought of as a tree-like model. We give the details explicitly for $\ZF$ and $\PA$ since these are the theories we are focusing on in this paper.

\begin{example}
Any model of either $\ZF$ and $\PA$ is tree-like as witnessed by the tree-like structures defined below.
\begin{enumerate}
\item
In the case of $\ZF$ define the tree-like by saying $x \in T$ if and only if $x$ codes a pair $( a, \alpha)$ where $\alpha$ is an ordinal and $a \subseteq V_\alpha$. We let $x \leq_T y$ if and only if $x$ codes $(a, \alpha)$, $y$ codes $(b, \beta)$ and $\alpha < \beta$ and $b \cap V_\alpha = a$. The linear order is the ordinals of the model and the ranking function is simply $r(a, \alpha) = \alpha$. Given a model $\calM \models \ZF$ we will refer to the version of this tree-like order defined in $\calM$ by $T^M_{\ZF}$.
\item
In the case of $\PA$ the tree-like order is simply the tree of finite binary sequences coded in the model. Given two such sequences $s, t$ we let $s \leq_T t$ just in case $t$ is an end extension of $s$. The linear order is the order of the model and the ranking function is just the length of the sequence. Given a model $\calM \models \PA$ we will refer to the version of this tree-like order defined in $\calM$ as $T^M_{\PA}$.
\end{enumerate}
\label{example1}
\end{example}
If $T$ is a theory for which every model can be construed as a tree-like model as in the above example we will refer to $T$ as a {\em tree-like theory}. Thus the above shows that $\PA$ and $\ZF$ are tree-like theories.

We finish this subsection with the introduction of one more idea. Given a tree-like model $\mathcal A$ we say that $\mathcal A$ satisfies the {\em collection scheme} if it satisfies the following scheme of sentences (which range of all formulas $\varphi(\vec{x})$):
\begin{equation*}
\forall x \in T \,\exists z \in T \, \forall t \in R(x) \, [\exists u \in T \, \varphi(t, u, ...) \to \exists u \in R(z) \, \varphi(t, u, ...)]
\end{equation*}
where $z$ does not occur in $\varphi(\vec{x})$. Recall here that $R(x)$ is the set of predecessors of $x$ in the tree ordering. In words the collection scheme asserts roughly that for all $x \in T$, if $\varphi(t, ...)$ has a witness for each $t \in R(x)$ then we can collect all of these witnesses together and they live in some bounded region of the tree. 

\begin{fact}[See Example 3.2 of \cite{KeislerTarskiSymp}]
Every model of $\ZF$ or $\PA$ satisfies the collection scheme with the tree-like structures described above.
\end{fact}

As mentioned in the introduction many of the results of this paper apply to theories all of whose models are tree-like and satisfy the collection scheme, however for definiteness we will primarily stick to $\ZF$ and $\PA$ from now on. 

\subsection{Basic Definitions for Models of $\PA$ and $\ZF$}

Importing the definitions from the previous section we here give an account of what we will need about model theory of $\ZF$ and $\PA$ for the rest of the paper. Throughout the rest of the paper, when discussing a model $\calM$ of $\ZF$ if we refer to an {\em ordinal in} $\calM$ we mean an element $a \in M$ so that $\calM \models$``$a$ is an ordinal" (and not necessarily an ordinal of the meta-theory). 

First let us note that for a cardinal $\kappa$ a model $\calM$ of $\PA$ is $\kappa$-like if it has size $\kappa$ but for all $a \in M$ $|[0, a]| < \kappa$. A model $\calM$ of $\ZF$ is $\kappa$-like if there are $\kappa$-many ordinals in $M$ but for every ordinal $\alpha \in M$ the set $V_\alpha$ (as defined in $\calM$) has size less than $\kappa$ and in particular the set of ordinals less than $\alpha$ has size less than $\kappa$.

\begin{definition}
Let $\kappa$ be a cardinal. 
\begin{enumerate}
\item
By arithmetizing the language of arithmetic and/or set theory we can think of $\calL$-formulas as coded computably by natural numbers. As such it makes sense to talk about a set of formulas as being e.g. computable, arithmetic etc. A model $\calM$ (of either $\PA$ or $\ZF$) is {\em recursively saturated} if it realizes every computable type with finitely many parameters.
\item
If $\calM$ is a model of $\PA$, then a {\em class} is a subset $A \subseteq M$ so that for all $a \in M$ the set $A \cap a := \{ b \in A \; | \; \calM \models b < a\}$ is definable in $\calM$ (parameters allowed). 
\item
If $\calM \models \ZF$ then a {\em class} is a subset $A \subseteq M$ so that for all ordinals $\alpha \in M$ we have that $A \cap V_\alpha := \{b \in A \; | \; \calM \models b \in V_\alpha \}$ is definable in $\calM$. In other words $(A \cap V_\alpha, \alpha) \in T^M_{\ZF}$.
\item
A model $\calM$ (of $\PA$ or $\ZF$) is {\em rather classless} if every class is definable.
\item
A $\kappa$-{Kaufmann Model} (of $\PA$ or $\ZF$) is a model $\calM$ which is $\kappa$-like, recursively saturated and rather classless. If $\kappa = \omega_1$ then we simply say $\calM$ is a {\em Kaufmann model}.
\end{enumerate}
\end{definition}

Note that a model of $\PA$ (respectively $\ZF$) is rather classless if and only if every branch of the tree defined in Example \ref{example1} is definable. In other words being rather classless is a particular example of being rather branchless.

The notion of a Kaufmann model can be defined for other theories as well and in fact every tree-like theory satisfying the collection scheme has Kaufmann models (with ``rather classless" replaced with ``rather branchless"). In particular there are Kaufmann models of $\PA$ and $\ZF$. As mentioned above, this was shown by Kaufmann in \cite{Kauf77} under the assumption that the combinatorial principle $\diamondsuit_{\omega_1}$ holds and the additional set theoretic assumption was eliminated by Shelah in \cite{Shelah2ndorderprop2}. Presently we recall a brief sketch of the existence of Kaufmann models for $\PA$ and $\ZF$ under the combinatorial principle $\diamondsuit_{\omega_1}$ as ideas from these arguments will be used repeatedly throughout the paper. Recall that $\diamondsuit_{\omega_1}$ is the statement that there is a sequence $\{A_\alpha \; | \; \alpha < \omega_1\}$ so that for all $\alpha< \omega_1$ $A_\alpha \subseteq \alpha$ and for every $A \subseteq \omega_1$ the set $\{\alpha \; | \; A \cap \alpha = A_\alpha\}$ is a stationary subset of $\omega_1$. A sequence such as $\{A_\alpha \; | \; \alpha < \omega_1\}$ as described above is called a $\diamondsuit$-sequence. From now on we will shorten the phrase ``the combinatorial principle $\diamondsuit_{\omega_1}$" to read simply ``$\diamondsuit$".

\begin{theorem}[Kaufmann \cite{Kauf77}]
If $\diamondsuit$ holds then every countable, recursively saturated model has an elementary end extension which is Kaufmann.
\label{ksthm}
\end{theorem}

Before sketching the proof we need to note a few things. First, recall that in the context of $\PA$, given two models $\calM$ and $\calN$ we say that $\calN$ is an {\em elementary end extension} of $\calM$, denoted $\calM \prec_{end} \calN$ if $\calM \prec \calN$ and for every $y \in N \setminus M$ and $x \in M$ we have $\calN \models x \leq y$, i.e. $(M, \leq)$ is an initial segment of $(N, \leq)$. The foundational MacDowell-Specker theorem states that every model of $\PA$ has an elementary end extension, see \cite[Theorem 2.2.8]{MOPAKS}. In the context of $\ZF$, a model $\calN$ is an elementary end extension of $\calM$, in symbols $\calM \prec_{end} \calN$, if $\calM \prec \calN$ and for all $x, y \in N$ with $\calN \models x \in y$, if $y \in M$ then $x \in M$. Briefly if $\calN$ adds no new elements to sets in $\calM$\footnote{The definitions for end extension for these two theories can be unified in the language of tree-like models. Namely, if $\calA$ is a tree-like model and $\calA \prec \calB$ (which will also be a tree like model) then $\calB$ end-extends $\calA$ if and only if for all $a \in T^A$ and $b \in T^B$ if $b \leq_{T_B} a$ then $b \in T^A$, i.e. $T^A$ is an initial segment of $T^B$, as a partial order.}. Analogues of the MacDowell-Specker theorem for $\ZF$ are more complicated, see \cite{MacDowellSpeckerZF} for more details. 


The proof of Theorem \ref{ksthm} uses the following lemma, which is also due to Kaufmann. 

\begin{lemma}[Kaufmann \cite{Kauf77}]
Let $\calM$ be a countable recursively saturated model of $\PA$ (respectively $\ZF$) and $A \subseteq M$. If $A$ is not definable, then there is a countable, recursively saturated model $\calN$ so that $\calM \prec_{end} \calN$ and $A$ is not coded into $\calN$ i.e. there is no $a \in N$ so that $a$ codes an $\calN$-finite sequence $s_a$ (respectively is an element of $N$) and $M \cap s_a = A$ (respectively $M \cap \{b \in N \; | \; \mathcal N \models b \in a\} = A$).
\label{Klem}
\end{lemma}

Both Lemma \ref{Klem}, and Theorem \ref{ksthm} are proved explicitly in \cite{Kauf77} for models of $\PA$ however, as Kaufmann remarks on \cite[p. 332]{Kauf77} they apply more generally to all tree-like models satisfying the collection scheme, with ``rather classless" replaced by ``rather branchless". In particular, the proofs apply mutatis mutandis to models of $\ZF$.

\begin{proof}[Proof of Theorem \ref{ksthm}]
Fix a countable recursively saturated model $\calM_0$ and a $\diamondsuit$ sequence $\vec{A} = \langle A_\alpha \; | \; \alpha < \omega_1\rangle$. We want to define a continuous chain $\langle \calM_\alpha \; | \; \alpha < \omega_1\rangle$ of countable, recursively saturated models so that $\calM_\alpha \prec_{end} \calM_{\alpha+1}$ for all $\alpha < \omega_1$ and the union of all the $\calM_\alpha$'s will be a Kaufmann model. This is done recursively. The universe of each model will be a countable ordinal. Note that there will necessarily be a club of $\delta < \omega_1$ so that $M_\delta = \delta$. \footnote{Here $M_\delta$ is the universe of $\mathcal M_\delta$, conforming to the convention mentioned in the first paragraph of this subsection} At limit stages we have to take unions since the chain is continuous so it remains to say what to do at successor stages. Suppose $\calM_\alpha$ has been defined. If $A_\alpha \subseteq M_\alpha$ is undefinable let $\calM_{\alpha+1}$ be as in Lemma \ref{Klem}, namely a countable, recursively saturated elementary end extension of $\calM_\alpha$ in which $A_\alpha$ is not coded. If $A_\alpha$ is not an undefinable subset of $M_\alpha$ (either because it's not a subset or because it's definable) then let $\calM_\alpha$ be {\em any} countable, recursively saturated elementary end extension of $\calM_\alpha$. This completes the construction.

Let $\calM = \bigcup_{\alpha < \omega_1} \calM_\alpha$. Clearly this model is an $\omega_1$-like, recursively saturated elementary end extension of $\calM_0$. The hard part is to show that it is rather classless. This is shown as follows: suppose $A \subseteq M$ is an undefinable class. It's straightforward to show that the set of $\alpha$ so that $A \cap M_\alpha$ is undefinable in $\mathcal M_\alpha$ is club, thus by $\diamondsuit$ there is an $\alpha$ so that $A \cap M_\alpha = A_\alpha$. But then $A \cap M_\alpha$ is not coded into $\calM_{\alpha+1}$ by our construction contradicting the assumption that $A$ is a class.
\end{proof}

The sequence above $\langle \calM_\alpha \; | \; \alpha < \omega_1\rangle$ is commonly called a {\em continuous, end-extensional filtration}. For short we will refer to such a sequence as simply a {\em filtration}.\footnote{Without the extra qualifiers this is not entirely standard, but this is the only type of filtration we will consider in this paper so no confusion will arise.}
\begin{definition}
A {\em filtration} is an $\omega_1$-length sequence $\langle M_\alpha \; |  \; \alpha < \omega_1\rangle$ of countable models so that for $\alpha < \beta$, $M_\alpha \prec_{end} M_\beta$ and for limit ordinals $\lambda < \omega_1$ $M_\lambda = \bigcup_{\xi < \lambda} M_\xi$. The filtration is said to be {\em recursively saturated} if every $M_\alpha$ is recursively saturated.
\end{definition}

We will need the notion of a (partial amenable) satisfaction class for models of arithmetic and set theory. This idea has generated an enormous amount of research and is central in the study of models of $\PA$ (and, to a lesser extent models of $\ZF$). We will only need a few facts, which we cherry pick below, and refer the reader to the excellent monograph \cite{MOPAKS} for more details. Unfortunately the definitions of $\PA$ and $\ZF$ are different enough that they have to be handled individually. We first present the more well-known case of $\PA$ and discuss its augmentation for $\ZF$. 

First, let us define an inductive, partial satisfaction class for a model of $\PA$. The definition we give, which comes from \cite{Sch81}, is not standard but it's easily seen that a model has a partial inductive satisfaction class in the sense below if and only if it has one in the sense defined e.g. in \cite[Definition 1.9.1]{MOPAKS}. Recall that for each standard $n < \omega$ there is (provably, in $\PA$ and $\ZF$) a $\Sigma_n$ formula $Tr_n(x, y)$ so that for all $\Sigma_n$ formulas $\varphi(z)$ $\PA \proves \forall y [\varphi(y) \leftrightarrow Tr_n (\varphi, y)]$. Given a model $\calM \models \PA$ let $W_n^M$ denote the set of pairs $(\varphi, a)$ so that $\varphi(x)$ is a $\Sigma_n$ formula with one free variable from the point of view of $\calM$ and $\calM \models Tr_n(\varphi, a)$ i.e. $\calM$ thinks that $a$ satisfies $\varphi$. 
\begin{definition}
Let $\calM$ be a model of $\PA$. A set $S \subseteq M^2$  is called a {\em partial inductive satisfaction class} if
\begin{enumerate}
\item
For all $x \in M$ $S_x := \{y \; | \; \langle x, y \rangle \in S\}$ is a set of pairs $(\varphi, a)$ so that $\varphi$ is a formula from the point of view of $\calM$ and $a \in M$.
\item
For all $n < \omega$ we have $S_n = W_n$.
\item
$(M, S)$ satisfies the induction scheme in the language expanded with a predicate for $S$.
\end{enumerate}
\label{satclass}
\end{definition}
Partial inductive satisfaction classes are the only types of satisfaction classes that will be discussed in this paper so we drop the qualifiers and refer to them simply as ``satisfaction classes". Note that the definition above is unchanged if we fix a nonstandard $a \in \calM$ and insist that for every $b \geq a$ the set $S_b = \emptyset$.

As mentioned above $\kappa$-Kaufmann models can be seen as a witness to incompactness at a cardinal $\kappa$. Schmerl has formalized this in the following striking way. 
\begin{theorem}[Schmerl, Theorem 3 of \cite{Sch81} ]
If there is a $\kappa$-Kaufmann model, then there is a $\kappa$-Aronszajn tree.

\label{schmerl} 
\end{theorem}

Roughly speaking the tree is the ``tree of attempts to build a satisfaction class".

\begin{proof}
Let $\calM$ be a $\kappa$-Kaufmann model. We will define a subset $T \subseteq M$ and a tree-like order on $T$ so that the levels of $T$ are indexed by the elements of $\calM$, $T$ has sequences of every order type in $\calM$, the set of such sequences in a given order type has size less than $\kappa$, and there is no subset $B \subseteq T$ in order type $\leq_M$. Clearly then any cofinal, well-founded subset of this \say{tree} will be a $\kappa$-Aronszajn tree. 

Fix $a \in M$ non-standard. Let $W_n$ denote the complete $\Sigma_n$-set (as defined in $\calM$), which we think of an an $\calM$-indexed list of $0$'s and $1$'s corresponding to its characteristic function on the set of pairs consisting of $\Sigma_n$ formulas and elements of $M$ (using some standard pairing function). The tree $T$ is the set of $b \in M$ so that there is a $d \in M$ and $b$ codes a $d \times a$ sized matrix whose entries are $0$ or $1$ and for which for each natural number $n < \omega$ the $n^{\rm th}$-column of $b$ is an $\calM$-finite initial segment of $W_n$. For elements $b_0, b_1 \in T$ coding matrices of size $d_0 \times a$ and $d_1 \times a$ respectively we let $b_0 \sqsubseteq_T b_1$ if $d_0 < d_1$ and $b_0 = b_1 \hook (d_0 \times a)$. In words, $b_0$ is below $b_1$ if and only if $b_1$ codes a larger matrix whose restriction to the coordinates $(d_0 \times a)$ is $b_0$ (end extend each column). This is clearly a tree like order, it remains to see that it forms a tree as described in the first paragraph.

First let's see that the levels have size ${<}\kappa$. Let $T_d :=\{b \in T \; | \; b \; {\rm codes \; a \; binary \; matrix \; of \; size} \; d \times a\}$. Then since $b \in T_d$ implies $b \in M$ and codes a sequence of size $d \times a$ there are at most $2^{d \times a}$ elements of $T_d$ (as computed in $\calM$) so by $\kappa$-likeness $|T_d| < \kappa$. 

Now lets see that the tree has height $\kappa$. This follows immediately by recursive saturation. For each $b \in T$ and $i < a$ let $b_n$ denote the $n^{\rm th}$ column of the matrix coded by $b$. For any $d \in M$ consider the type $p_d(x) := \{\exists e > d \; ( x {\rm \; codes \; a \; matrix \; of \; size \;} e \times a)\} \cup \{ x_n \subseteq W_n \; | \; n < \omega\}$. Clearly this is a finitely consistent, recursive type so it has a realization in $\calM$. But any such realization is an element of height greater than $d$.

Finally there is no cofinal branch. This follows by rather classlessness: from any cofinal branch we could define a satisfaction class by the definition of the tree, but since any satisfaction class is undefinable this can't exist. See \cite[Lemma 4.1]{Sch81} for a more detailed discussion of this last point. Note that if $\kappa$ is an uncountable regular cardinal then any class is inductive, see \cite[pp. ~258-259]{MOPAKS}.
\end{proof}

As Schmerl notes, what the proof above shows is that if $\kappa$ has the tree property then every $\kappa$-like recursively saturated model has a satisfaction class. Regardless of the properties of the order type of $\calM$, the proof shows that given any recursively saturated model $\calM$, there is an associated tree $T^M_{sat}$ whose levels are cofinal in the model. Moreover, if $\calM$ is $\kappa$-like for some regular $\kappa$ then $T^M_{sat}$ has a cofinal branch if and only if $\calM$ has a satisfaction class. We will call such a tree the {\em satisfaction tree for} $\calM$ (relative to $a$). 

Now let us handle the case of models of $\ZF$. For any defined term $t$ in $\ZF$, if $\calM \models \ZF$ let $t^\calM$ denote the corresponding term in $\ZF$ e.g. $\omega^\calM$, $V_\alpha^\calM$ etc. Theorem \ref{schmerl} holds almost verbatim for models of $\ZF$, however this seems to be folklore and we could not find a proof so we write out the details here. First we need an augmentation of Definition \ref{satclass}. Recall that if $\calM \models \ZF$ then for any integer $n \in \omega^\calM$ we can define in the model a truth predicate for $\Sigma_n$-truth i.e. the class $\{(\varphi(x), a) \; | \; \varphi(x)$ is $\Sigma_n$ and $\calM \models$``$\varphi(a)$ holds"$\}$. Let us call this class $W_n^M$ in analogy with the arithmetic case. If $\alpha \in M$ is an ordinal in $M$ then let $W_n^M \hook \alpha$ consist of the set of all $(\varphi(x), a) \in W_n^M$ so that $a \in V_\alpha^\calM$. Note that this is a set from the point of view of $\calM$.

\begin{definition}
Let $\calM \models \ZF$. A subset $S \subseteq \omega^\calM \times M$ is a {\em partial amenable satisfaction class} if the following hold.
\begin{enumerate}
\item
For all $x \in M$ $S_x := \{y \; | \; \langle x, y\rangle \in S\}$ is a set of pairs $(\varphi, a)$ so that $\varphi$ is a formula from the point of view of $\calM$ and $a \in M$.
\item
For all standard $n < \omega$ we have $S_n = W^M_n$.
\item
$(M, S)$ satisfies the replacement and comprehension schemes in the language expanded with a predicate for $S$.
\end{enumerate}

\end{definition}
Note that by a diagonal argument similar to the classic one used in Tarski's undefinability of truth no partial amenable satisfaction class can be definable. As in the case of $\PA$ this definition is not standard. However, in the terminology of \cite[Definition 2.2 d)]{Enayat21}, if $\calM$ is $\omega$-nonstandard then for any $a \in \omega^\calM$ nonstandard if $S$ is an $a$-satisfaction class, then the set $\{(k, \varphi(x), a)\; | \; \varphi(a) \in S \; {\rm is} \; \Sigma_k\}$ is a partial amenable satisfaction class in our sense and, conversely, by overspill, if $S$ is a partial amenable satisfaction class as defined above then there is a nonstandard $a \in \omega^\calM$ so that $\bigcup_{b < a} S_b$ is an $a$-satisfaction class in the sense of \cite[Definition 2.2, d)]{Enayat21}. Therefore, for $\omega$-nonstandard models the existence of a partial amenable satisfaction class is equivalent to the existence of an $a$-satisfaction class for some nonstandard $a$.

Let us now explain how to define the analogue of $T^M_{sat}$ for a model of $\ZF$. The construction of the tree is enough to imply that Schmerl's Theorem \ref{schmerl} holds for models of $\ZF$. In any case it is the construction of the tree that we will need moving forward.

Let $\calM \models \ZF$ and for each $n \in \omega^\calM$, let $\Sigma_n^\calM$ be the set of $\Sigma_n$-formulas of $\calL_{\ZF}$ as defined in $\calM$. In an abuse of notation, for each infinite $\alpha \in ON^\calM$ and standard $n < \omega$ let us associate $W^M_n\hook \alpha$ with its characteristic function $\chi_n:\Sigma_n^\calM \times V_\alpha^\calM \to 2$. Now, for each ordinal $\alpha$ in $\calM$ let $T_\alpha^M$ consist of all $t \in M$ so that $t \in \omega^\calM \times M^2$ so that for all $n \in \omega^\calM$ standard we have that $t_n = W^M_n\hook \alpha$ and for all $n \in \omega^\calM$ we have that $t_n$ is a function mapping $\Sigma^\calM_n \times V_\alpha^\calM \to 2$. Let $T^M_{sat} = \bigcup_{\alpha \in ON^\calM} T_\alpha^M$. If $t, s \in T^M_{sat}$ we let $t\leq_{sat} s$ if and only if $t \in T^M_\alpha$, $s\in T^M_\beta$ with $\alpha < \beta \in ON^\calM$ and for each $n \in \omega^\calM$ we have that $s_n \hook \Sigma_n^\calM \times V_\alpha = t_n$. In other words, in each column, the restriction of $s$ to parameters in $V^\calM_\alpha$ is exactly $t$. Clearly this is a tree like order with a ranking function in the ordinals of $\calM$.

Now an essentially verbatim proof to the one given for Theorem \ref{schmerl} shows that if $\calM$ is recursively saturated then $T^M_{\alpha}$ is non-empty for every $\alpha \in ON^\calM$, if $\calM$ is $\kappa$-like for some regular cardinal $\kappa$ then $|T^M_{\alpha}| < \kappa$ for each $\alpha \in ON^\calM$ and any cofinal branch through $T^M_{sat}$ codes a partial amenable satisfaction class for $\calM$ so if $\calM$ is rather classless then $T^M_{sat}$ has no cofinal branch. In total, if $\calM$ is a $\kappa$-Kaufmann model then $T^M_{sat}$ is a $\kappa$-Aronszajn tree.

The tree $T^M_{sat}$ will be discussed in Sections 4 and 5. There it will not matter whether we are discussing models of $\PA$ or $\ZF$ since, in light of the above discussion these ideas can be defined equally for both. Therefore, given a Kaufmann model $\calM$ of either $\PA$ or $\ZF$ we will define $T^M_{sat}$ to mean the corresponding tree depending on the theory without much further comment.

Fix a Kaufmann model $\calM$ of any tree-like theory. Suppose $\mathbb P$ is a forcing notion, when does $\forces_\mathbb P ``\check{\calM}$ is not Kaufmann" ? Obviously, if $\mathbb P$ collapses $\aleph_1$ to be countable, then $\omega_1$-likeness is killed. What about if $\mathbb P$ does not collapse $\omega_1$? This motivates the following definition.
\begin{definition}
A Kaufmann model $\calM$ is {\em destructible} if there is an $\omega_1$-preserving forcing notion $\mathbb P$ so that $\forces_\mathbb P ``\check{\calM}$ is not Kaufmann".
\end{definition}

In this language, an immediate corollary of Main Theorem \ref{mainthm1} is the following.
\begin{corollary}
The existence of destructible Kaufmann models for $\PA$ and $\ZF$ is independent of $\ZFC$.
\end{corollary}

 Before ending this section, let us make one observation about destructibility of Kaufmann models that will guide the rest of the paper. Suppose $\calM$ is a Kaufmann model of $\PA$ or $\ZF$ and $\mathbb P$ is an $\omega_1$-preserving forcing notion. Then in $V^\mathbb P$ $\calM$ is still $\omega_1$-like, and by absoluteness, there cannot be any new recursive types, so $\calM$ is still recursively saturated. Therefore, if $\mathbb P$ kills the Kaufmann-ness of $\calM$ it's because it added an undefinable class. This is what we will use to kill Kaufmann models.

Finally, let us note that some similar ideas to those presented here were previously explored by Enayat in \cite{Enayatstandardmodel}. In particular, in Theorem 4.2 of that paper Enayat observes that there are rather classless models of $\ZFC^- + V = H_{\aleph_1}$ which remain rather classless in any forcing extension preserving $\omega_1$. Thus in the language of this paper Enayat shows that there is always an indestructible model of $\ZFC^- + V = H_{\aleph_1}$.

\section{Killing Destructible Kaufmann Models}
In this section we prove the first part of Main Theorem \ref{mainthm1}. Specifically we show the following, which is much more general.

\begin{theorem}
Assume $\MAone$. Let $\calA = \langle A, T^A, \leq_T, O^A, \leq_O, r, P_1, ...\rangle$ be a tree-like model which is $\omega_1$-like and rather branchless. If $\P$ is a forcing notion so that $\forces_\P$``$\check{\calA}$ is not rather branchless" then $\P$ collapses $\omega_1$.

In particular there are no destructible Kaufmann models of any tree-like theory satisfying the collection scheme.
\label{MAthm}
\end{theorem}
For the ``in particular" part in the case of $\PA$ or $\ZF$ note that since forcing cannot add elements to old models, adding a class to a model $\calM$ of either $\PA$ or $\ZF$ is equivalent to adding a cofinal branch to $T^M_{\PA}$ or $T^M_{\ZF}$.

The rest of this section is devoted to proving Theorem \ref{MAthm}. Fix an $\omega_1$-like, rather branchless tree-like model $\calA = \langle A, T^A, \leq_T, O^A, \leq_O, r, P_1, ...\rangle$ for the rest of the section. Note that since $\calA$ is $\omega_1$-like, $T^A$ has uncountably many levels but each level is countable. We will show that in any forcing extension if $\calA$ has a new class, then $\omega_1$ is collapsed. To begin we need a few more definitions about trees.


\begin{definition}
Let $T= \langle T, \leq_T\rangle$ be a tree-like order.
\begin{enumerate}
\item
If $T$ is $\omega_1$-like we say that $T$ is {\em Aronszajn} if it has no uncountable, linearly ordered subset. 
\item
We say that $T$ is {\em special} if there is a function $f:T \to \omega$ so that if $x\leq_T y$ then $f(x) \neq f(y)$. 
\item
We say that $T$ is {\em weakly special} if there is a function $f:T \to \omega$ so that if $x \leq_T y, z$ and $f(x) = f(y) = f(z)$ then $y$ and $z$ are comparable in the $\leq_T$ ordering.
\end{enumerate}
\end{definition}

The application of $\MAone$ needed to prove Theorem \ref{MAthm} is the following fact, due to Baumgartner, Malitz and Reinhardt.

\begin{fact}[Theorem 4 of \cite{BMR70}]
Assume $\MAone$. Let $T$ be an $\omega_1$-like tree-like partial order of size $\aleph_1$. If $T$ is Aronszajn then $T$ is special.
\label{specialma}
\end{fact}

We also need the following, well known fact.
\begin{lemma}
Suppose $T$ is an $\omega_1$-like tree-like order. If $T$ is weakly special then any forcing adding a cofinal branch collapses $\omega_1$.
\label{coll}
\end{lemma}

\begin{proof}
Suppose $f:T \to \omega$ witnesses that $T$ is weakly special, $\mathbb P$ is a forcing notion and $\forces_\mathbb P `` \dot{b} \subseteq \check{T}$ is a new, cofinal branch". Let $G \subseteq \mathbb P$ be generic over $V$ and let $b = \dot{b}_G$. We claim that (in the extension) for each $n < \omega$ the set $f^{-1}(\{n\}) \cap b$ is bounded. Note that this implies the lemma since we will have that $b$, which is a set of size $\aleph_1^V$ can be covered by countably many countable sets.

To see the claim, suppose for some $n < \omega$ we have that $p \forces \check{f}^{-1}``(\{\check{n}\}) \cap \dot{b}$ is unbounded". By strengthening if necessary, we may assume that $p$ decides some $x \in T$ is in $\dot{b}$ and $f(x) = n$. Now since $\dot{b}$ is forced to be new there are incompatible extensions $p_0$ and $p_1$ of $p$ and incompatible elements $x_0$ and $x_1$ extending $x$ so that for $i < 2$ $p_i \forces x_i \in \dot{b}$ and $f(x_i) = n$. But this contradicts the defining property of $f$. 
\end{proof}


\begin{lemma}
$\MA_{\aleph_1}$ implies that $T^A$ is weakly special.
\label{finws}
\end{lemma}

The proof of this lemma uses the fact that if the conclusion of Fact \ref{specialma} holds then any tree of cardinality $\aleph_1$ with at most $\aleph_1$ many uncountable branches is weakly special. For (well-founded) trees, this result is well known, see \cite[Corollary 7.8]{BaumPFA}. The proof goes through verbatim for ranked trees $T^A$ which appear as the tree-like order of an $\omega_1$-like tree-like model $\calA$, but we give the details below for the sake of completeness, as well as to present the proof to model theorists who may not be as familiar with these ideas as set theorists.

\begin{proof}
Since $\calA$ is rather branchless $T^A$ has only $\aleph_1$-many branches. Enumerate all the uncountable branches by $\mathcal B = \{b_\alpha \; | \; \alpha < \omega_1\}$. Fix an injection $g: \mathcal B \to T^A$ so that for each $\alpha$ $g(b_\alpha) \in b_\alpha$. By \cite[Lemma 7.6]{BaumPFA}, one can choose $g$ so that whenever $g(b_\alpha) <_{fin} g(b_\beta)$ then $g(b_\beta) \notin b_\alpha$. Now let $S = \{t \in T^A \; | \; \forall b \in \mathcal B \, {\rm if} \; t \in b \, {\rm then} \; t \leq_{T} g(b)\}$. This is a tree-like order with the order inherited from $T^A$. Moreover, it's uncountable since it contains the range of $g$. It has no uncountable branches. To see this, towards a contradiction, suppose that $b$ were an uncountable branch through $S$. Let $\bar{b} = \{t \in T \; | \; \exists s \in b \; t <_{T} s\}$ i.e. the downward closure of $b$ in $T^A$. This must be an uncountable branch through $T$. But then since $g(\bar{b}) \in \bar{b}$ we get an $s \in b$ with $g(\bar{b}) \leq_{T} s$ contradicting the definition of $S$.

Applying Fact \ref{specialma}, $\MA_{\aleph_1}$ implies that $S$ is special. Let $f:S \to \omega$ be such a specializing function. Let $t\in T^A \setminus S$. We extend $f$ to include $t$ as follows. Since $t \notin S$ there is a branch $b$ so that $t \in b$ but $g(b) \leq_{T} t$. This branch is unique: If $g(b_\alpha) <_{T} g(b_\beta) <_{T} t$ with $t\in b_\alpha \cap b_\beta$ then in particular $g(b_\beta) \in b_\alpha$ which contradicts the choice of $g$. Now let $f(t) = f(g(b))$ for this unique branch $b$. 

\begin{claim}
$f:T \to \omega$ has the property that if $f(s) = f(t) = f(u)$ and $s \leq_T t, u$ then $t$ and $u$ are comparable, i.e. it witnesses that $T^A$ is weakly special.
\end{claim}

\begin{proof}
Let $s \leq_{T} t, u$ be as in the claim. Since $f(t) = f(s)$ at least one of $t$ and $s$ is not in $S$ since $f$ is injective on chains in $S$. In fact neither $s$ nor $t$ are in $S$ unless $s = g(b)$ for some $b$. To see this, first note that if $s \in S$ then, since $t \notin S$ we would have that there is some $b$ so that $b$ is the unique branch with $t \in b$ and $g(b) \leq_{T} t$ and, since $s \in b$ as well and $s \in S$ we have that $s \leq_{T} g(b)$ and so either $s = g(b)$ or $f(s) \neq f(g(b)) = f(t)$ which is a contradiction. Similarly if $t \in S$ then since $s \notin S$ there is some branch $c$ so that $s \in c$ but $g(c) \leq_{T} s$ and since $g(c), t \in S$ and $g(c) <_{T} t$ we have that $f(g(c)) \neq t$ but this is a contradiction since $f(g(c)) = f(s) = f(t)$.

Now, let $b$ be the unique branch so that $t \in b$ and $g(b) \leq_{T} t$. As noted before, $s \in b$ as well. If $s <_{T} g(b)$ then there is a branch $c \neq b$ so that $s \in c$ and $g(c) \leq_T s$ (since either $s = g(c)$ or is above it, by the argument in the previous paragraph). But now $g(c), g(b) \in S$ and $g(c) <_{T} g(b)$ so $f(g(c)) \neq f(g(b))$ but this is a contradiction since $f(s) = f(g(c))$ and $f(t) = f(g(b))$. Therefore $g(b) \leq_{T} s$, $b = c$ and hence $s \in b$. A symmetric argument allows one to conclude the same for $u$ so $t, s, u \in b$ and hence are comparable.
\end{proof}

Since the claim is proved the lemma is as well.
\end{proof}

Let's now conclude the proof of Theorem \ref{MAthm}. 

\begin{proof}[Proof of Theorem \ref{MAthm}]
Assume $\MAone$. If $\P$ forces that $\calA$ is not rather branchless then $\P$ adds a branch to $T^A$. But by Lemma \ref{finws} we have that $T^A$ is weakly special and hence by Lemma \ref{coll} $\P$ collapses $\omega_1$.
\end{proof}

Before moving on to the proof of the second part of Main Theorem \ref{mainthm1}, let's observe some easy extensions of Theorem \ref{MAthm}. These involve the following two observations from the proof: first was that we did not need $\MA_{\aleph_1}$ only that every Aronszajn, $\omega_1$-like tree-like partial order of cardinality $\aleph_1$ which embeds into an $\omega_1$-like tree-like partial order with countable levels is special and second is that we didn't use the fact that $\calA$ was rather branchless, only that it had ${\leq}\aleph_1$-many classes. Therefore we actually have the following result which gives a stronger conclusion from a weaker hypothesis.
\begin{theorem}
Assume every tree-like Aronszajn order which embeds into an $\omega_1$-like ranked tree is special. If $\calA = \langle A, T^A, \leq_T, O^A, \leq_O, r, P_1, ...\rangle$ is an $\omega_1$-like tree-like model so that $T^A$ has ${\leq} \aleph_1$-many uncountable branches, then there is no $\omega_1$-preserving forcing adding a branch to $T^A$. 

In particular, if $\calM \models \PA$ is $\omega_1$-like, $T^M_{\PA}$ has $\leq \aleph_1$-many classes and every Aronszajn subtree of $T^M_{\PA}$ is special then no forcing notion can add a class to $\calM$ without collapsing $\omega_1$. The same is true for $\calM \models \ZF$ with $T^M_{\PA}$ replaced by $T^M_{\ZF}$.
\end{theorem}

Using the forcing of \cite{Switz20widetrees}, the above hypothesis can be forced over a model of $\CH$ without adding reals so it's consistent with $\CH$ that there are no destructible Kaufmann models. In fact the following is consistent.
\begin{corollary}
If $\mathsf{ZF}$ is consistent then so is $\ZFC + \CH$ and for all $\omega_1$-like tree-like models $\calA = \langle A, T^A, \leq_T, O^A, \leq_O, r, P_1, ...\rangle$, if $T^A$ has ${\leq}\aleph_1$-many branches then there is no $\omega_1$-preserving forcing adding a branch to $T^A$. In particular if $T$ is a consistent completion of $\PA$ or $\ZF$ and $\ZF$ is consistent then so is $\ZFC + \CH$ and for all $\omega_1$-like models $\calM \models T$ with ${\leq}\aleph_1$-many classes there is no $\omega_1$ preserving forcing adding a class to $\calM$.
\end{corollary}

Finally let us note that if there is an $\omega_1$-like tree-like model with more than $\aleph_1$-many branches then the tree is a Kurepa tree. Since it's consistent (relative to an inaccessible) that there are no Kurepa trees it's consistent that there is no $\omega_1$-preserving forcing notion adding a class to any $\omega_1$-like tree-like model.
\begin{corollary}
If $\mathsf{ZF}$ plus ``there is an inaccessible cardinal" is consistent, then $\ZFC$ plus ``any forcing notion adding a branch to an $\omega_1$-like tree-like model collapses $\omega_1$" is consistent both with $\CH$ and the negation of $\CH$. 

In particular, if $\mathsf{ZF}$ plus ``there is an inaccessible cardinal" is consistent then $\ZFC$ both with $\CH$ and its negation are consistent with ``no $\omega_1$-like model of $\PA$ or $\ZF$ can have a class added to it by forcing without collapsing $\omega_1$".
\end{corollary}

\begin{proof}
By what has been said it suffices to note that from an inaccessible, $\MAone$ can be forced alongside the failure of Kurepa's hypothesis and (for the $\CH$ case) a countable support iteration of the main forcing from \cite{Switz20widetrees} of length $\kappa$ for $\kappa$ inaccessible plus some routine bookkeeping works.
\end{proof}

Note that it was observed by Keisler \cite{KeislerTarskiSymp} that if there is a Kurepa tree, then there is a model $\calM$ so that $T^M_{\PA}$ is Kurepa, so the inaccessible is needed.

\section{Building a Destructible Kaufmann Model}
In this section we prove the second part of Main Theorem \ref{mainthm1}. Specifically we show the following.
\begin{theorem}
Assume $\diamondsuit$. Then (for $\ZF$ and $\PA$) there is a Kaufmann model $\calM$ so that the satisfaction tree $T^M_{sat}$ contains a Souslin subtree and hence is destructible.
\label{diamondthm}
\end{theorem}

The ``hence" part follows by observing that forcing with the Souslin tree is ccc, and therefore $\omega_1$-preserving, but the generic branch will define a satisfaction class for $\calM$ as explained in the proof of Theorem \ref{schmerl}. The idea behind the proof is to use the diamond sequence to weave together Kaufmann's original argument for the existence of a Kaufmann model with Jensen's classic argument of the existence of a Souslin tree.

\begin{proof}
Fix a diamond sequence $\vec{A} = \langle A_\alpha \; | \; \alpha < \omega_1\rangle$, a countable, recursively saturated model $\calM_0$ and, if $\calM_0 \models \PA$ (as opposed to $\ZF$), a nonstandard element $a \in M_0$. Other than this one sentence, the proof is verbatim the same whether we work with $\ZF$ or $\PA$ so we remain ambiguous from now on. The only thing we really need is the construction of the satisfaction tree as in Theorem \ref{schmerl} and the proceeding discussion of its analogue for models of $\ZF$. 

As in the proof of Theorem \ref{ksthm}, we will build a filtration of countable, recursively saturated models $\langle \calM_\alpha \; | \; \alpha < \omega_1\rangle$ however this time we will also build a $\subseteq$-increasing continuous sequence of sets $\langle S_\alpha \; | \; \alpha < \omega_1\rangle$ so that for all $\alpha < \omega_1$ we have $S_\alpha \subseteq M_\alpha$, and $\calM = \bigcup_{\alpha < \omega_1} \calM_\alpha$ is a Kaufmann model and $S := \bigcup_{\alpha < \omega_1} S_\alpha$ is a Souslin subtree of $T^M_{sat}$ relative to $a$. 

We construct $(\calM_\alpha, S_\alpha)$ recursively. The construction essentially mirrors Kaufmann's original construction of a Kaufmann model from $\diamondsuit$ done at the same time as Jensen's original construction of a Souslin tree from $\diamondsuit$. Given any $\calM_\alpha$ let $T^\alpha_{sat}$ be the satisfaction tree for $\calM_\alpha$ relative to $a$. We already gave $\calM_0$, let $S_0$ be $T_{sat}^0$. Assume that we have constructed $(\calM_\xi, S_\xi)$ for all $\xi < \alpha$, and that for each $\xi < \alpha$ $\calM_\xi$ is a countable recursively saturated end extension of is predecessors, $S_\xi$ is a subset of $T^\xi_{sat}$ which intersects every level $d \in M_\xi$ and so that each $t \in S_{\xi}$ has extensions on all levels above it. Without loss, we can assume that each $M_\xi$ is a set of countable ordinals. As before there will be a club of $\xi$ so that $M_\xi = \xi$.

\noindent \underline{Case 1}: $\alpha$ is a limit ordinal. By the requirements we have, $\calM_\alpha = \bigcup_{\xi < \alpha} \calM_\xi$ and $S_\alpha = \bigcup_{\xi < \alpha} S_\xi$.

\noindent \underline{Case 2}: $\alpha = \beta + 1$ for some $\beta$. If $A_\beta \subseteq M_\beta$ is an undefinable class then extend $\calM_\beta$ as in Lemma \ref{Klem} so that $A_\beta$ is not coded into $\calM_\alpha$. Otherwise let $\calM_\alpha$ be any countable, recursively saturated end extension. Note the priority: we have $\calM_\alpha$ now and will use it to define $S_\alpha$.

If $A_\beta \subseteq S_\beta$ is a maximal antichain, then do as follows. First choose a level $b \in M_\alpha \setminus M_\beta$ and, for each of the countably many $t \in S_\beta$ choose exactly one $s_t \in A_\beta$ comparable with $t$ and one element $u_{s_t, t} \in T^\alpha_{sat}$ on the $b^{\rm th}$ level that extends both $s_t$ and $t$. Note that by the maximality of $A_\beta$ there is such an $s$ for each $t$ and by recursive saturation in $M_\alpha$ there is such a $u_{s_t, s}$. The set of all such $u_{s_t, t}$ will be the $b^{\rm th}$ level of the Souslin tree we're constructing. Specifically, let $S_\alpha^- = S_\beta \cup \{u_{s_t, t} \; | \; t \in S_\beta\}$ and let $S_\alpha$ be the downward closure of $S^-_{\alpha}$ in $T_{sat}^\alpha$ alongside every extension of an element in $S^-_{\alpha}$ in $T^\alpha_{sat}$ to the levels $b' > b$ in $M_\alpha$. 

If $A_\beta$ is not a maximal antichain of $S_\beta$ then let $S_\alpha$ be simply the collection of all extensions in $T^\alpha_{sat}$ of every node in $S_\beta$ to every level in $M_\alpha \setminus M_\beta$. This completes the construction. 

Let $\calM= \bigcup_{\alpha < \omega_1} \calM_\alpha$ and let $S = \bigcup_{\alpha < \omega_1} S_\alpha$. The verification that $\calM$ is Kaufmann is verbatim as in Theorem \ref{ksthm}.

To see that $S$ is a Souslin tree, suppose that $A \subseteq S$ is a maximal antichain. I claim that there is a club of $\xi$ so that $A\cap S_\xi$ is a maximal antichain in $S_\xi$. Let $C$ denote the set of all such $\xi$. Clearly $C$ is closed since any increasing union of maximal antichains will again be a maximal antichain. To see that $C$ is unbounded, fix an ordinal $\xi:=\xi_0$. If $A \cap S_\xi$ is not maximal then, for each of the countably many elements $t \in S_\xi$ not comparable with anything in $A \cap S_\xi$ find some element $a_t \in A$ which is comparable with them. Let $\xi_1 > \xi_0$ be such that $A \cap S_{\xi_1}$ contains all of these $a_t$ ($\xi_1$ is countable since there are only countably many things to add). Continuing in this way, recursively define for each $n < \omega$ a countable ordinal $\xi_{n+1} > \xi_{n}$ so that every $a \in A \cap S_{\xi_n}$ is comparable with something in $A \cap S_{\xi_{n+1}}$. Finally let $\xi_\omega := {\rm sup}_{n \in \omega}\xi_n$. Clearly $A \cap S_{\xi_\omega}$ is maximal by the continuity requirement of the construction.

It follows by $\diamondsuit$ that there is an $\xi$ so that $A_\xi = A \cap S_\xi$. But then there is a level $d \in M_{\xi + 1}$ so that every element of $A_\xi$ is comparable with a node $t$ of height $d$ by our construction so if $s \in T^M_{sat}$ is of height greater than $d$ then $s \notin A$ since it's comparable with some node in $A_\xi \subseteq A$. Thus $A$ is bounded and therefore countable.
\end{proof}
As a remark, let us note that the above construction can also be done via forcing: let $(\calM_0, S_0, a)$ be as above and $\mathbb P$ be the set of pairs $(\calM, S_M)$ so that $\calM_0 \prec \calM$, $\calM$ is recursively saturated, countable, $S_0 \subseteq S_M$ and $S_M \subseteq T_{sat}^M$, which has non-empty intersection with every level in $M$. The order is pairwise by elementary end extension and end extension as a partially ordered set. This forcing is countably closed and the verification that it adds a destructible Kaufmann model goes through exactly as in the proof of the theorem, replacing the $\diamondsuit$ construction by a collection of density arguments. I do not know if the second coordinate is necessary or if forcing with the models alone will make the resulting generic Kaufmann model destructible, though I suspect that this is the case. However, this forcing construction is weaker than the proof from diamond since the forcing, being countably closed and adding a subset to $\omega_1$, adds a diamond sequence.

Finally we note that even though $\diamondsuit$ implies $\CH$, it's consistent that there are destructible Kaufmann models and the continuum is arbitrarily large.

\begin{proposition}
Assume $\diamondsuit$, then there is a destructible Kaufmann model in the extension by any number of Cohen reals. 
\end{proposition}

\begin{proof}
Suppose $\calM$ and $S$ are as in the proof of Theorem \ref{diamondthm} (the $\diamondsuit$ hypothesis guarantees their existence). Let $\mathbb P$ be the forcing to add $\lambda$ many Cohen reals for your favorite $\lambda$. Since $\mathbb P$ is ccc, it preserves $\omega_1$ hence $\calM$ remains an $\omega_1$-like recursively saturated model. Moreover, Cohen forcing neither kills Souslin trees nor adds branches to $\omega_1$-trees (like $T^{M}_{\PA}$/$T^M_{\ZF}$) so it cannot add a class to $\calM$ nor kill the Souslin-ness of $S$. Hence $\calM$ is still a Kaufmann model and $S$ is still a ccc forcing adding a satisfaction class.
\end{proof}


\section{Axiomatizability of Kaufmann Models}
In this section we prove Main Theorem \ref{mainthm2}. The proof involves the logic, $L_{\omega_1, \omega} (Q)$ the infinitary logic $L_{\omega_1, \omega}$ enriched with the quantifier $Q$ where the interpretation of $Qx\varphi(x)$ is ``there exist uncountably many $x$ so that $\varphi(x)$ holds". Recall from \cite{L(Q)} that a {\em standard model} of $L_{\omega_1, \omega}(Q)$ is a structure $\mathcal M = \langle M, [M]^{\geq \omega_1}, ...\rangle$ so that for any formula $\varphi(\bar{x}, y)$ and any $\bar{a} \in M^{ln(\bar{x})}$ we have that $\mathcal M \models Qy \varphi(\bar{a}, y)$ if and only if the set $\{y \in M \; | \; \mathcal M \models \varphi(\bar{a}, y)\}$ is uncountable. There is also a relatively straightforward, Hilbert-style notion of proof for this logic, see \cite[p. 69]{L(Q)}. In \cite[Theorem 4.10]{L(Q)} Keisler proved the following completeness theorem.

\begin{theorem}[Keisler]
For any sentence of $L_{\omega_1, \omega}(Q)$ $\psi$ we have that $\vdash \psi$ if and only if for every standard model $\mathcal M$ in the vocabulary of $\psi$ we have $\mathcal M \models \psi$.
\end{theorem}

Note that this theorem implies that if an $L_{\omega_1, \omega}(Q)$ sentence from $V$ has a model in some forcing extension, then it has one in the ground model via generic absoluteness. This is the key step in Shelah's argument that there are Kaufmann models in $\ZFC$. One thing to note is that formulas of $L_{\omega_1, \omega}(Q)$ are coded by reals so in forcing extensions adding reals, one adds new formulas.

I will need the following, elementary observation.
\begin{observation}
Suppose $\calM$ is an $\calL$-structure for some $\calL$ and $\mathbb P$ is a forcing notion that preserves $\omega_1$. Then for any $L_{\omega_1, \omega}(Q)$ formula $\psi(\bar{x})$ and any tuple $\bar{a}$ in $\calM$ we have that $\calM \models \psi(\bar{a})$ if and only if $\forces_{\mathbb P}$``$\check{\calM} \models \psi(\bar{a})$". 
\label{absoluteness}
\end{observation}
Roughly this observation amounts to saying that $L_{\omega_1, \omega}(Q)$ truth cannot be changed by $\omega_1$-preserving forcing. 

\begin{proof}
The proof is by induction on $\psi$. Since $L_{\omega_1, \omega}$ satisfaction is absolute between forcing extensions and grounds the only non obvious case is when $\psi$ is of the form $Qx\varphi(x, \bar{y})$. However, this follows immediately by the inductive hypothesis and the fact that $\mathbb P$ preserves $\omega_1$. 
\end{proof}

Using these results and the proofs of Theorems \ref{MAthm} and \ref{diamondthm} we will show the following.
\begin{theorem}
Let $T$ be any consistent completion of either $\PA$ or $\ZF$. Let $\calL$ be the language of $T$.
\begin{enumerate}
\item
Under $\MA_{\aleph_1}$ there is an $L_{\omega_1, \omega}(Q)$ sentence $\psi_T$\footnote{Of course, $\psi_T$ depends on $T$.} in the language $\calL$ enriched with a single unary function symbol $f$, $\calL (f)$, so that a model $\calM \models T$ is Kaufmann if and only if there is an expansion of $\calM$ to an $\calL(f)$-structure satisfying $\psi_T$. 
\item
Under $\diamondsuit$ there is a Kaufmann model $\calM \models T$ so that given any expansion $\calL'$ of the language $\calL$, and any expansion of $\calM$ to an $\calL'$-structure, $\calM '$, and any countable set $X$ of $L_{\omega_1, \omega}(Q)$ sentences in the signature $\calL'$ there is a model $\calN$ which agrees with $\calM '$ about the truth of every sentence in $X$ but carries a satisfaction class for its $\calL$-reduct. In particular, the $\calL$-reduct of $\calN$ is not rather classless.
\end{enumerate}
\label{mainthm22}
\end{theorem}

\begin{remark}
The wording of Part 2 is a little verbose. The point is that, even enriching $\calM$ with any amount of extra structure, we can always find a model which agrees with $\calM$ on any $L_{\omega_1, \omega}(Q)$ sentence and has a satisfaction class. Thus, in contrast to the case under $\MAone$, no amount of extra structure suffices to axiomatize Kaufmann models in $L_{\omega_1, \omega}(Q)$.
\end{remark}

As mentioned in the introduction Part 1 of the above theorem can be inferred easily from the proof of \cite[Theorem 6]{Shelah2ndorderprop2}. I'm not sure if this was observed at the time. We give a complete, self contained proof here however for the convenience of the reader. Note that in the proof we will often write sentences in the signature of $\calL_{\PA}$ (respectively $\calL_{\ZF}$) involving natural numbers, $n < \omega$. By this we will always mean the formal term $\underline{n} := S(S(...(S(0))...))$ (with $n$ iterations of the successor function $S$) if working in $\calL_{\PA}$ or in $\calL_{\ZF}$ the formal term for the von Neumann ordinal $n$. Since every model of $\PA$ or $\ZF$ contains a copy of the natural numbers there is no ambiguity in this.

\begin{proof}[Proof of Part 1 of Theorem \ref{mainthm22}]
Assume $\MAone$ holds. For the sake of definiteness we will prove the theorem for a consistent, completion $T$ of $\PA$ however, replacing $T^M_{\PA}$ and all related vocabulary with that of $T^M_{\ZF}$ proves the case of $\ZF$. In fact the proof of this part is valid for any tree-like theory satisfying the collection scheme which is formalizable in a computable language\footnote{However, if the theory does not allow for definitions of each standard natural number like $\PA$ and $\ZF$ do then constants for each one need to be added to the language as is done in \cite[Theorem 6]{Shelah2ndorderprop2}. Also, the restriction to computable languages is not needed except to make sense of the notion of ``recursively saturated".}. 

First observe that if $\calM \models T$ then one can easily write down being $\omega_1$-like and recursively saturated in $L_{\omega_1, \omega}(Q)$ (in any language) as follows. 
\begin{enumerate}
\item
$\calM$ is $\omega_1$-like if and only if it satisfies $Qx (x = x) \land \forall y \neg Qx (x \leq y)$
\item
 $\calM$ is recursively saturated if and only if it satisfies 
\begin{center}
$\forall \bar{y} \bigwedge_{p(x, \bar{y}) \; {\rm a \; computable \; type}}( \bigwedge_{\Phi(x, \bar{y}) \; {\rm finite \; subset \; of} \; p(x, \bar{y})}\exists x \Phi(x, \bar{y}) \rightarrow \exists x \bigwedge_{\phi(x, \bar{y}) \in p(x, \bar{y})} \phi(x, \bar{y}))$
\end{center} 
\end{enumerate}

Therefore, what we need to show is that there is a sentence $\psi$ in the language $\mathcal L_{\PA}(f)$ so that a model $\calM$ is rather classless if and only if there is a function $f^M:M \to M$ so that $\langle M,..., f^M\rangle \models \psi$. The idea is that $f$ will be a weak specializing function for the tree $T^M_{\PA}$ (which exists by $\MA_{\aleph_1}$) and using this function we will be able to say that all uncountable branches are definable. Shelah's sentence from \cite{Shelah2ndorderprop2} says more or less the same, though because we're not working in the general setup he works in there we can simplify things slightly. First note that $f$ being weakly special can be expressed as follows\footnote{Note that to prove the theorem for models of $\ZF$ we replace the term $\leq_{\PA}$ in the following displayed sentence, and the one after, with $\leq_{\ZF}$.}:

\begin{equation*}
ES(f) := \forall x [\bigvee_{n < \omega} f(x) = n] \land \forall s, t, u [f(s) = f(t) = f(u) \land s \leq_{\PA} t, u \rightarrow (t \leq_{\PA} u \lor u \leq_{\PA} t)]
\end{equation*}

So it remains to show that, for an weakly specializing function $f$, we can write down that $f$ witnesses that $\calM$ is rather classless. The sentence is as follows, below ``$RC$" means ``rather classless": 

\begin{center}
$RC(f) := \forall s \bigvee_{n < \omega} (f(s) = n  \land Q t (f(s) = f(t) = n \land s \lneq_{fin} t) )\rightarrow \exists \bar{a} \bigvee_{\varphi \in \mathcal L_{\PA}} [ \forall y \varphi(y, \bar{a}) \leftrightarrow \exists t (s \leq_{\PA} t \land t(y) = 1 \land f(t) = n)]$
\end{center}
Note that since elements of $T^M_{\PA}$ are (coded) binary sequences the notation ``$t(y) = 1$" makes sense\footnote{In the case of models of $\ZF$ we replace ``$t(y) = 1$"  by ``$y \in t$".}. The reader should convince themself that in English the above says the following:

\begin{center}
``For all $s$, if for some $n$ $f(s) = n$ and there are uncountably many $t$ so that $s \leq_{\PA} t$ and $f(t) = n$ then there are an $\bar{a}$ and a formula $\varphi \in \mathcal L_{\PA}$ so that for all $y$ $\varphi(y, \bar{a})$ if and only if $t(y) = 1$ for some $t$ with $s \leq_{\PA} t$ and $f(t) = n$."
\end{center}

Since, by the proof of Theorem \ref{MAthm}, we know that every Kaufmann model's $T^M_{\PA}$ is weakly special, we need to show that $\calM$ is Kaufmann if and only if its weakly specializing function $f$ satisfies $RC(f)$. Here are the details. First suppose that $\calM$ is an $\omega_1$-like, recursively saturated model of $\PA$ which has an expansion to $\mathcal L_\PA(f)$ satisfying $ES(f) \land RC(f)$. Fix such an $f^M:M \to M$.  Let $b$ be an uncountable branch through $T^M_{\PA}$. We need to show that there is a formula $\varphi$ and a tuple $\bar{a}$ so that for all $y \in M$, $\cup b(y) = 1$ if and only if $M \models \varphi(y, \bar{a})$. By $RC(f^M)$ then there is an $\bar{a}$ and a formula $\varphi \in \mathcal L_{\PA}$ so that for all $y$ $\varphi(y, \bar{a})$ if and only if there is a $t$ above $s$ with $t(y) = 1$ and $f(t) = n$. By the property of weak specializing functions, if $s\leq_{\PA} t $ and $f(t) = n$ then $t \in b$. Therefore $\cup b(y) = 1$ if and only if $\varphi(y, \bar{a})$ as required.

For the converse, suppose $\calM$ is a Kaufmann model and let $f^M$ be a weak specializing for $T^M_{\PA}$ (which exists by $\MAone$). We claim that this $f^M$ satisfies $RC(f)$. To see this, fix $s \in T^M_{\PA}$ and $n < \omega$ and suppose that $f^M(s) = n$ and there are uncountably many $t$ above $s$ in $T^M_{\PA}$ with $f^M(t) = n$. Then the set of these $t$ must generate a cofinal branch $b$ by weak specialness so we can define that branch as $\cup b(y) = 1$ if and only if  $M \models \varphi(y, \bar{a})$ by rather classlessness, hence $RC(f)$ is satisfied.
\end{proof}

Before continuing on to the proof of Part 2, let me comment on the relation between this proof and Shelah's \cite[Theorem 6]{Shelah2ndorderprop2}. This theorem, despite being foundational in the field seems to have been very rarely written down aside from in the original article. Restricted to the case of Kaufmann models of $\PA$ and $\ZF$, Shelah's proof shows much the same as what is shown above. The difference is that he replaces the application of $\MAone$ by a concrete use of a ccc forcing to specialize $T^M_{\PA}$. As a result his proof shows (in our language) that every $\mathcal L_\PA$ reduct of a model of $ES(f) \land RC(f)$ is Kaufmann (this is identical to the backward direction above) and, for every Kaufmann model $\calM$ there is a ccc forcing extension of $V$ in which $\calM$ has an expansion to a model of $ES(f) \land RC(f)$ (using $\MAone$ instead of forcing this is the forward direction). By composing this result with Theorem \ref{ksthm} Shelah gets that every model of set theory has a forcing extension in which there is a model of $ES(f) \land RC(f)$. By Keisler's completeness theorem it follows that in $V$ this sentence is consistent and hence has a model. But then that model's reduct to $\mathcal L_\PA$ (respectively $\calL_{\ZF}$) is Kaufmann thus proving that $\ZFC$ suffices to prove the existence of Kaufmann models. A natural question is whether the detour through forcing extensions was necessary in this argument. Part 2 will show that, at least sometimes, the answer is ``yes".

\begin{proof}[Proof of Part 2 of Theorem \ref{mainthm22}]
Let $T$ be a consistent completion of $\PA$ or $\ZF$ and let $\calL$ be its language. Let $\calM \models T$ be the model constructed in the proof of Theorem \ref{diamondthm} and let $S$ be the Souslin subtree of $T^M_{sat}$. The existence of this model is the only application of $\diamondsuit$. Let $\mathcal L' \supseteq \mathcal L$ be any language extending $\calL$ and $\calM'$ be any expansion of $\calM$ to an $\calL'$ structure. We need to show that there is an $\calN$ which agrees with $\calM'$ on any countably many ${L_{\omega_1, \omega}(Q)}$ sentences but whose $\calL$-reduct has a satisfaction class (and hence is not rather classless). 

Since $S$ is Souslin, the $L_{\omega_1, \omega}(Q)$ theory of $\calM$ is the same in $V$ as in any generic extension of $V$ by $S$ by Observation \ref{absoluteness} plus the fact that, since any Souslin tree is $\omega$-distributive, $S$ won't add new reals and hence it won't add new $L_{\omega_1, \omega}(Q)$ sentences either. Let $G \subseteq S$ be generic and work in $V[G]$. In this model, the branch $G$ codes a satisfaction class $A_G$ for $\calM$. Consider a new theory, $T'$ in the language $\mathcal L'$ enriched with a unary predicate $A$ giving the $L_{\omega_1, \omega}(Q)$ theory of $\calM'$ in $\mathcal L'$ plus ``$A$ is a satisfaction class". This theory is consistent, since $\langle M', ...,  A_G\rangle$ is a model and, moreover, it is in $V$ since it's the union of a theory in $V$ with a simple set of additional sentences, definable in any model of set theory. Since consistency is absolute between models of set theory with the same natural numbers, $V \models$``$T'$ is consistent". Hence by Keisler's completeness theorem, any countable subtheory $\bar{T} \subseteq T'$ has a model $\calN$ with a satisfaction class. Consider the reduct of $\calN$ to $\mathcal L'$. This model is exactly what we wanted so the proof is complete.

\end{proof}

It's tempting to conclude in the above proof that $\calN$ can be made to be fully $L_{\omega_1, \omega}(Q)$ equivalent to $\calM'$ but Keisler's theorem is sentence by sentence and since $L_{\omega_1, \omega}$ lacks a compactness theorem, it's not clear that this conclusion can be made, hence the restriction to countable subtheories. I'm not sure whether the stronger conclusion is consistent or not, though I suspect that it is.


\section{Conclusion and Open Questions}
There remain many open questions in this area. I want to finish this paper by listing some. The most interesting is the following.
\begin{question}
Is there a non-forcing-theoretic characterization of destructible Kaufmann models? Is this related to some sort of resplendency or something truth theoretic?
\end{question}

Regarding the construction of destructible Kaufmann models by forcing:
\begin{question}
Does forcing with countable, recursively saturated models ordered by end extension add a Kaufmann model whose satisfaction tree is Souslin (not just having a Souslin subtree)?
\end{question}

Also, it's worth asking:
\begin{question}
What tree types can a satisfaction tree take? In particular, can the satisfaction tree for a Kaufmann model be Souslin (and not just contain a Souslin subtree)? What about trees types for trees of the form $T^M_{\PA}$/$T^M_{\ZF}$?
\end{question}

This paper is not the first to consider strong logics in the context of Kaufmann models. Surprisingly though the following appears to be open.
\begin{question}
Which logics extending $L_{\omega, \omega}$ can axiomatize Kaufmann models provably in $\ZFC$? Consistently?
\end{question}

Finally, while this entire discussion has concerned $\aleph_1$-Kaufmann models, there seems to be a wealth of possible directions in studying general $\kappa$-Kaufmann models. Note that by Schmerl's Theorem \ref{schmerl}, if $\kappa$ has the tree property then there are no $\kappa$-Kaufmann models. The converse of this appears to be open.

\begin{question}
 Does the statement ``there are no $\aleph_2$-Kaufmann models" imply the tree property on $\aleph_2$? What is the consistency strength of ``there are no $\aleph_2$-Kaufmann models"?
\end{question}

\bibliography{mopabib}
\bibliographystyle{plain}

\end{document}